\documentclass[11pt,a4paper]{amsart}
\usepackage{amssymb,amsmath}
\usepackage[dvips]{color}
\usepackage[dvips,final]{graphics}
\usepackage{epstopdf}
\usepackage{amscd}

\usepackage{pdfsync}

\theoremstyle{plain}
\newtheorem{theorem}{Theorem}[section]

\newtheorem{corollary}[theorem]{Corollary}

\newtheorem{lemma}[theorem]{Lemma}
\theoremstyle{definition}
\newtheorem{definition}[theorem]{Definition}

\theoremstyle{remark}
\newtheorem{remark}[theorem]{Remark}

\begin{document}

\bibliographystyle{plain}

\newcommand{\Rn}{\mathbb R^n}
\newcommand{\E}{\mathbb E}
\newcommand{\Rm}{\mathbb R^m}
\newcommand{\rn}[1]{{\mathbb R}^{#1}}
\newcommand{\R}{\mathbb R}
\newcommand{\C}{\mathbb C}
\newcommand{\G}{\mathbb G}
\newcommand{\M}{\mathbb M}
\newcommand{\Z}{\mathbb Z}
\newcommand{\D}[1]{\mathcal D^{#1}}
\newcommand{\Ci}[1]{\mathcal C^{#1}}
\newcommand{\Ch}[1]{\mathcal C_{\mathbb H}^{#1}}
\renewcommand{\L}[1]{\mathcal L^{#1}}
\newcommand{\BVG}{BV_{\G}(\Omega)}
\newcommand{\supp}{\mathrm{supp}\;}

\newcommand{\dom}{\mathrm{Dom}\;}
\newcommand{\Ast}{\un{\ast}}

\newcommand{\average}{{\mathchoice {\kern1ex\vcenter{\hrule height.4pt
width 6pt depth0pt} \kern-9.7pt} {\kern1ex\vcenter{\hrule height.4pt width
4.3pt depth0pt}
\kern-7pt} {} {} }}
\newcommand{\ave}{\average\int}

\newcommand{\hhd}[1]{{\mathcal H}_d^{#1}}
\newcommand{\hsd}[1]{{\mathcal S}_d^{#1}}

\newcommand{\he}[1]{{\mathbb H}^{#1}}
\newcommand{\hhe}[1]{{H\mathbb H}^{#1}}

\newcommand{\cov}[1]{{\bigwedge\nolimits^{#1}{\mfrak g}}}
\newcommand{\vet}[1]{{\bigwedge\nolimits_{#1}{\mfrak g}}}

\newcommand{\covw}[2]{{\bigwedge\nolimits^{#1,#2}{\mfrak g}}}

\newcommand{\vetfiber}[2]{{\bigwedge\nolimits_{#1,#2}{\mfrak g}}}
\newcommand{\covfiber}[2]{{\bigwedge\nolimits^{#1}_{#2}{\mfrak g}}}

\newcommand{\covwfiber}[3]{{\bigwedge\nolimits^{#1,#2}_{#3}{\mfrak g}}}

\newcommand{\covv}[2]{{\bigwedge\nolimits^{#1}{#2}}}
\newcommand{\vett}[2]{{\bigwedge\nolimits_{#1}{#2}}}

\newcommand{\covvfiber}[3]{{\bigwedge\nolimits^{#1}_{#2}{#3}}}

\newcommand{\vettfiber}[3]{{\bigwedge\nolimits_{#1,#2}{#3}}}

\newcommand{\covn}[2]{{\bigwedge\nolimits^{#1}\rn {#2}}}
\newcommand{\vetn}[2]{{\bigwedge\nolimits_{#1}\rn {#2}}}
\newcommand{\covh}[1]{{\bigwedge\nolimits^{#1}{\mfrak h_1}}}
\newcommand{\veth}[1]{{\bigwedge\nolimits_{#1}{\mfrak h_1}}}
\newcommand{\hcov}[1]{{_H\!\!\bigwedge\nolimits^{#1}}}
\newcommand{\hvet}[1]{{_H\!\!\bigwedge\nolimits_{#1}}}

\newcommand{\covf}[2]{{\bigwedge\nolimits^{#1}_{#2}{\mfrak h}}}
\newcommand{\vetf}[2]{{\bigwedge\nolimits_{#1,#2}{\mfrak h}}}
\newcommand{\covhf}[2]{{\bigwedge\nolimits^{#1}_{#2}{\mfrak h_1}}}
\newcommand{\vethf}[2]{{\bigwedge\nolimits_{#1,#2}{\mfrak h_1}}}
\newcommand{\hcovf}[2]{{_H\!\!\bigwedge\nolimits^{#1}_{#2}}}
\newcommand{\hvetf}[2]{{_H\!\!\bigwedge\nolimits_{#1,#2}}}

\newcommand{\defin}{\stackrel{\mathrm{def}}{=}}
\newcommand{\gradh}{\nabla_H}
\newcommand{\current}[1]{\left[\!\left[{#1}\right]\!\right]}
\newcommand{\scal}[2]{\langle {#1} , {#2}\rangle}
\newcommand{\escal}[2]{\langle {#1} , {#2}\rangle_{\mathrm{Euc}}}
\newcommand{\Scal}[2]{\langle {#1} \vert {#2}\rangle}
\newcommand{\scalp}[3]{\langle {#1} , {#2}\rangle_{#3}}
\newcommand{\dc}[2]{d_c\left( {#1} , {#2}\right)}
\newcommand{\res}{\mathop{\hbox{\vrule height 7pt width .5pt depth 0pt
\vrule height .5pt width 6pt depth 0pt}}\nolimits}
\newcommand{\norm}[1]{\left\Vert{#1}\right\Vert}
\newcommand{\modul}[2]{{\left\vert{#1}\right\vert}_{#2}}
\newcommand{\perh}{\partial_\mathbb H}

\newcommand{\ccheck}{{\vphantom i}^{\mathrm v}\!\,}

\newcommand{\wcheck}{{\vphantom i}^{\mathrm w}\!\,}

\newcommand{\mc}{\mathcal }
\newcommand{\mbf}{\mathbf}
\newcommand{\mfrak}{\mathfrak}
\newcommand{\mrm}{\mathrm}
\newcommand{\no}{\noindent}
\newcommand{\dis}{\displaystyle}

\newcommand{\U}{\mathcal U}
\newcommand{\ga}{\alpha}
\newcommand{\gb}{\beta}
\newcommand{\gga}{\gamma}
\newcommand{\gd}{\delta}
\newcommand{\eps}{\varepsilon}
\newcommand{\gf}{\varphi}
\newcommand{\GF}{\varphi}
\newcommand{\gl}{\lambda}
\newcommand{\GL}{\Lambda}
\newcommand{\gp}{\psi}
\newcommand{\GP}{\Psi}
\newcommand{\gr}{\varrho}
\newcommand{\go}{\omega}
\newcommand{\gs}{\sigma}
\newcommand{\gt}{\theta}
\newcommand{\gx}{\xi}
\newcommand{\GO}{\Omega}

\newcommand{\Wedge}{\buildrel\circ\over \wedge}

\newcommand{\WO}[4]{\mathop{W}\limits^\circ{}\!_{#4}^{{#1},{#2}}
(#3)}

\newcommand{\GH}{H\G}
\newcommand{\N}{\mathbb N}

%

\newcommand{\Nhmin}{N_h^{\mathrm{min}}}
\newcommand{\Nhmax}{N_h^{\mathrm{max}}}

\newcommand{\Mhmin}{M_h^{\mathrm{min}}}
\newcommand{\Mhmax}{M_h^{\mathrm{max}}}

\newcommand{\un}[1]{\underline{#1}}

\newcommand{\curl}{\mathrm{curl}\;}
\newcommand{\curlh}{\mathrm{curl}_{\he{}}\;}
\newcommand{\hd}{\hat{d_c}}
\newcommand{\divg}{\mathrm{div}_\G\,}
\newcommand{\divgh}{\mathrm{div}_{\hat\G}\,}
\newcommand{\divh}{\mathrm{div}_{\he{}}\,}
\newcommand{\e}{\mathrm{Euc}}


\newcommand{\hatcov}[1]{{\bigwedge\nolimits^{#1}{\hat{\mfrak g}}}}

\title[
] {H\"older regularity of viscosity solutions of some fully nonlinear equations in the Heisenberg group
}

\author{Fausto Ferrari}
\address{Dipartimento di Matematica dell'Universit\`a di Bologna, Piazza di Porta S. Donato, 5, 40126 Bologna, Italy.}
\email{\tt fausto.ferrari@unibo.it}

\thanks{The author is supported by  MURST, Italy, and
GNAMPA project 2017: {\it Regolarit\`a delle soluzioni viscose per equazioni a derivate parziali non lineari degeneri}}

\thanks{}
%

%

%
\keywords{Heisenberg group, viscosity solutions, Theorem of the sums.}

\subjclass{35D40, 35B65, 35H20.
}

\begin{abstract} 
 In this paper we prove the H\"older regularity of bounded, uniformly continuous,  viscosity solutions  of some degenerate fully nonlinear equations in the Heisenberg group. 
\end{abstract}

\maketitle

\tableofcontents

\section{Introduction}

In this paper we prove a result concerning the regularity of viscosity solutions of some degenerate fully nonlinear elliptic equations in the Heisenberg group that are uniformly continuous in the Heisenberg group. 

Indeed, it is well known that the theory of viscosity solutions is very flexible and that the existence of viscosity solutions of second order PDEs is not strictly related to the degeneracy of the elliptic operator, see \cite{userguide}, \cite{Crandall}. In addition, the regularity of viscosity solutions of second order elliptic, possibly nonlinear, PDEs  is traditionally faced by proving, as a first step, by proving the Harnack inequality. On the other hand, the  proof of the Harnack inequality  is based on the Alexandroff-Bakelman-Pucci inequality and the consequent maximum principle, see \cite{GT} and \cite{CC}. In the case of subelliptic structures we recall also the following contributions: \cite{GuLa}, \cite{GuMo}, \cite{DoGaNi}, \cite{Baetaltri}.

We are interested in the regularity of viscosity solutions of that fully nonlinear equations that are not uniformly elliptic in the classical sense. In order to be more precise we mention here essentially the case in which we are interested in: nonlinear PDEs that are modeled the vector fields  belonging to the first layer of a stratified algebra.  The simplest example, is given by the Heisenberg group. This type of operators belongs to a class of operator studied in \cite{BM}, where a comparison result has been proved.

In our aim we would like to prove a regularity results, possibly simply a modulus of continuity, for viscosity solutions,  without using the Harnack inequality. In this paper we prove that bounded viscosity solutions (uniformly continuous) of some fully nonlinear second order PDEs on all of $\mathbb{R}^n,$ modeled on the vector fields of the first stratum of the Lie algebra of the simplest Heisenberg group, are H\"older continuous. 
Our result is heavely based on the theorem due to Crandall, Ishii and Lions, see \cite{userguide}, very often called in literature,  {\it Theorem of the sums}. Due to its importance in our approach we will recall it in a specific section.  It is worth to say that among the PDEs that we will deal with,  we find also cases whose solutions are much more regular than the result we are able to prove. For instance they could be smooth, indeed. Nevertheless, among the family of the equations that we consider, in the worst case, they are at least H\"older regular.

Since, we introduce our result for a genuine class of fully non linear operators build on the vector fields of the first stratum of the Heisenberg group, 
for people that were not habit to this language, the self-contained Section \ref{preliminaryCarnot} is dedicated to cover the main definitions that we will use in the paper. Anyhow, a comprehensive discussion of the subject can be found in many handbooks, see for instance \cite{BLU} and \cite{CDT}. 

One of the key points of our approach lies on the structure of the operator that we are considering. Namely, instead of the Hessian matrix $D^2u$ we will use matrices obtained by the product between a non-negative matrix associated with the fields of the first stratum of the Lie algebra of the group and the classical Hessian matrix. This product produces the intrinsic horizontal Hessian that we define in Section \ref{preliminaryCarnot}. However for describing in this introduction the result we anticipate some quite known information about the simplest Heisenberg group.

The Heisenberg group $\mathbb{H}^1\equiv \mathbb{R}^3,$ is endowed with the Lie algebra $\mathfrak{g}=\mathfrak{g}_1\oplus\mathfrak{g}_2$ where   $\mathfrak{g}_1\equiv\mbox{span}\{X_1,X_1\}$ and
$\mathfrak{g}_2=\mbox{span}\{[X_1,X_2]\}$ as vector spaces. Since $X_1(x)=\frac{\partial}{\partial x_1}+2x_2\frac{\partial}{\partial x_3}$ and $X_2=\frac{\partial}{\partial x_1}-2x_1\frac{\partial}{\partial x_3}$
it is possible to define the matrix $P(x)=\sigma^T(x)\sigma(x),$ where the rows of $\sigma(x)$ are determined by the vector fields of the first stratum, that is: 
\begin{equation}
\sigma(x)=\left[
\begin{array}{ll}
1,&0,2x_2\\
0,&1,-2x_1
\end{array}
\right],
\end{equation}
where $x=(x_1,x_2,x_3)\in \mathbb{H}^1.$
Instead of the Hessian matrix $D^2u(x)$ defined for a sufficiently smooth function $u,$ we are leaded to consider the
$3\times 3$ matrix
$$
\sqrt{P(x)}D^2u(x)\sqrt{P(x)},
$$
that preserves indeed, for every $x\in \mathbb{H}^1,$ the trace of the matrix $P(x)D^2u(x).$ In this way we preserve the sub-Laplace operator of $u$ on the group given by  $X_1^2u+X_2^2u.$ More precisely,

\begin{equation*}\begin{split}
&\mbox{tr}(\sqrt{P(x)}D^2u(x)\sqrt{P(x)})\\
&=\mbox{tr}(P(x)D^2u(x))=\Delta_{\mathbb{H}^1}u(x)=X_1^2u+X_2^2u=\mbox{tr}(D^{2,*}u(x)),
\end{split}
\end{equation*}
where $D^{2,*}u(x)$ denotes the symmetrized horizontal Hessian $2\times2$ matrix in the Heisenberg group. 

We remark here that the choice of the matrix
$$
\sqrt{P(x)}D^2u(x)\sqrt{P(x)}
$$
seems important even if it is possible to consider others approach. Indeed, our strategy works only using that matrix instead of  $P(x)D^2u(x),$ that in general is not symmetric. Anyhow the same type of result can be obtained using the horizontal Hessian matrix $D^{2,*}u(x)$ that has a different dimension with respect to the classical Hessian matrix. In both cases we have to  read the information contained respectively in those second order objects, only recalling the Theorem of the sums. We shall come back later on this topic.

Our main result can be stated in the general framework of a family of fully nonlinear degenerate operators originated from the vectors fields of the first stratum of the algebra of a Carnot groups.  
In particular, in the Heisenberg group $\mathbb{H}^1,$ we define the following fully nonlinear operators.

\begin{definition}\label{definition1H}
Let $\Omega\subseteq\mathbb{R}^3$ be an open set.  Let $\lambda, \Lambda$ be positive real numbers such that $0<\lambda\leq \Lambda.$ 
Let $P=\sigma^T\sigma$ be the non-negative matrix of variable coefficients, where 
\begin{equation}
\sigma(x)=\left[
\begin{array}{ll}
1,&0,2x_2\\
0,&1,-2x_1
\end{array}
\right].
\end{equation}
For every function $G:S^3\to \mathbb{R}$ such that  for every $A,B\in S^3,$ if $B\leq A$ then 
$$
\lambda\mbox{Tr}(A-B)\leq G(A)-G(B)\leq \Lambda \mbox{Tr}(A-B),
$$
we define the function $F:S^3\times\Omega\to \mathbb{R}$ such that  for every $M\in S^3$ and for every $x\in \Omega$
$$
F(M,x)=G(\sqrt{P(x)}M\sqrt{P(x)}).
$$

In addition, for these operators, we define the following class of  fully  nonlinear equations in the open set $\Omega\subset\mathbb{H}^1\equiv\mathbb{R}^3$
 \begin{equation}\label{defvisco}
 F(D^{2}u(x),x)-c(x)u(x)=f(x),
 \end{equation}
 where $c\in C^{0,\beta}(\Omega),$  $f\in C^{0,\beta'}(\Omega),$ $\beta,\beta'\in (0,1],$ and $c\geq 0$ for every $x\in \Omega.$ 
\end{definition}

 We remark that these operators are not contained in the classical class of fully nonlinear operators that are uniformly elliptic, see \cite{CC} for the definition.
 
 Indeed, defining
 $$
 \mathcal{A}_{\lambda,\Lambda}=\{A\in S^3:\quad \lambda|\xi|^2\leq \langle A\xi,\xi\rangle\leq \Lambda |\xi|^2\}
 $$
 where $0<\lambda\leq \Lambda,$
 in our class we find  the linear sub-Laplace operator $\Delta_{{\mathbb{H}^1}},$ that does not belong to the class of uniformly elliptic operators classically defined in \cite{CC} as well as we do not find the following extremal operators:
 $$
 \mathcal{P}^+_{\mathbb{H}^1,\lambda,\Lambda}(M,x)-c(x)u,\quad \mathcal{P}^-_{\mathbb{H}^1,\lambda,\Lambda}(M,x)-c(x)u,
 $$
 where
\begin{equation*}
\begin{split}
 \mathcal{P}^+_{\mathbb{H}^1,\lambda,\Lambda}(M,x):=\max\{\mbox{Tr}(A\sqrt{P(x)}M\sqrt{P(x)}):\quad A\in \mathcal{A}_{\lambda,\Lambda}\},
 \end{split}
 \end{equation*}
 and
 \begin{equation*}
\begin{split}
 \mathcal{P}^-_{\mathbb{H}^1,\lambda,\Lambda}(M,x):=\min\{\mbox{Tr}(A\sqrt{P(x)}M\sqrt{P(x)}):\quad A\in \mathcal{A}_{\lambda,\Lambda}\}.
 \end{split}
 \end{equation*}

We come back now on the class of the fully nonlinear operators that we have introduced, because there is another approach to point out. 

Indeed we can use the stratified structure of the Lie algebra using only the intrinsic object in the Heisenberg group. For instance the definition of class of our intrinsic fully nonlinear operator can be done in the following way.
For every $A\in S^3$ and for every $x\in \Omega \subseteq\mathbb{H}^1,$ let us define, see Section \ref{preliminaryCarnot},
\begin{equation}
\tilde{A}_{x}=\left[
\begin{array}{lr}
\langle A X(x),X(x)\rangle,&\langle A X(x),Y(x)\rangle\\
\langle A X(x),Y(x)\rangle,&\langle A Y(x),Y(x)\rangle
\end{array}
\right]
\end{equation}

\begin{definition}\label{definition2H}
 Let $\lambda, \Lambda>0.$ Let $F:S^2\to \mathbb{R}$ be a continuos function such that for every $H_1,H_2\in S^2$  
 if $H_1\geq H_2,$ then
 $$
\lambda \mbox{Tr}(H_1-H_2)\leq F(H_1)-F(H_2)\leq \Lambda \mbox{Tr}(H_1-H_2). 
$$
Let $\Omega\subseteq\mathbb{H}^1,$ for every $A\in S^3$ and for every $x\in \Omega$ we define $\tilde{F}(A,x)=F(\tilde{A}_x)$
 \end{definition}
 This definition fits with the definition given for the classical fully nonlinear operators except than for the fact that $\tilde{A}\in S^2,$ $x\in \mathbb{H}^1\equiv\mathbb{R}^3$ instead of $\tilde{A}\in S^3,$  and $x\in \mathbb{R}^3.$
 
 We will come back in Section \ref{preliminaryCarnot} about the intrinsic horizontal symmetrised Hessian matrix $D^{2,*}u(x)$ at the point $x\in\Omega.$ For the reader that does not know this definition it is sufficient to focus the attention to the fact that even if $u$ is defined in an opens subset of $\mathbb{R}^3$ the matrix $D^{2,*}u(x)$ is symmetric and of $2\times 2$ order, instead of $3\times3$ order. In addition $D^{2,*}u(x)$  may be defined using the vector fields of the first stratum of the Lie algebra in the Heisenberg group applied two times and taking in account that they do not commute,
 since
 $$
 \tilde{D^2u(x)}=D^{2*}_{\mathbb{H}^1}u(x).
 $$
 
We remark that to this class of operators still applies the classical definition of viscosity solution, see Section \ref{preliminary}. 

The class of operators $\tilde{F}$ satisfying the Definition \ref{definition2H}
are not 
 uniformly elliptic operator in the classical sense, see \cite{CC}.

 We postpone some details about these operators, however it is clear that they include as very particular case  the real part of the Kohn Laplace operator in the Heisenberg group, namely:
$$
\Delta_{\mathbb{H}^1}u:=(\frac{\partial}{\partial x_1}+2x_2\frac{\partial}{\partial x_3})^2u+(\frac{\partial}{\partial x_1}-2x_1\frac{\partial}{\partial x_3})^2u,
$$
that is degenerate elliptic in every point of $\mathbb{H}^1.$ Indeed, the smallest eigenvalue of the nonnegative matrix 
\begin{equation}
P(x)=\left[
\begin{array}{lll}
1,&0,&2x_2\\
0,&1,&-2x_1\\
x_2,&-2x_1,&4(x_1^2+x_2^2)
\end{array}
\right],
\end{equation}
 where $x:=(x_1,x_2,x_3)\in \mathbb{H}^1\equiv \mathbb{R}^3,$ is always $0$ and $\Delta_{\mathbb{H}^1}u(x)=\mbox{Tr}(P(x)D^2u(x))=\mbox{div}(P(x)\nabla u(x)).$ In this special case \begin{equation*}\begin{split}
 \tilde{F}(D^2u(x),x)&=\mbox{Tr}(P(x)D^2u(x))=Tr(D^{2,*}u(x))=X^2u(x)+Y^2u(x)\\
 &=F(D^{2,*}u(x)),
 \end{split}\end{equation*}
 where as usual $X=\frac{\partial}{\partial x_1}+2x_2\frac{\partial}{\partial x_3}$ and $Y=\frac{\partial}{\partial x_1}-2x_1\frac{\partial}{\partial x_3}.$
   
   We remark one more time that $D^2u(x)$
is a $3\times 3$ matrix, while $D^{2*}u(x)$ is a $2\times2$ matrix. 

 As a consequence, given $0<\lambda\leq \Lambda$  numbers,  it is quite natural define also the extremal operators 
 $$
\tilde{\mathcal{P}}^+_{\mathbb{H}^1,\lambda,\Lambda}(D^{2*}u(x)):=\max_{a\in \tilde{a}_{\lambda,\Lambda}}\mbox{Tr}(aD^{2,*}u(x))=\Lambda \sum_{e>0}e-\lambda \sum_{e<0}e 
 $$
 and 
 $$
\tilde{\mathcal{P}}^-_{\mathbb{H}^1,\lambda,\Lambda}(D^{2*}u(x)):=\min_{a\in \tilde{a}_{\lambda,\Lambda}}\mbox{Tr}(aD^{2,*}u(x))=\lambda \sum_{e>0}e-\Lambda \sum_{e<0}e,
 $$
where $$\tilde{a}_{\lambda,\Lambda}=\{a\in S^2:\quad \lambda |\xi|^2\leq \langle a\xi,\xi\rangle\leq \Lambda |\xi|^2,\:\:\xi\in \mathbb{R}^2\setminus\{0\}\},$$  and $e$ denotes the generic  eigenvalue of the symmetrized horizontal Hessian matrix of $u$ at $x.$

We are in position to  state our main result.

\begin{theorem}\label{mainregularityresult}
 Let $u\in C(\mathbb{H}^1)$ be a bounded uniformly continuous function that is a viscosity solution of the equation
 $$F(D^{2}u(x),x)-c(x)u(x)= f(x),\quad \mathbb{H}^1,$$ and $F$ is an operator satisfying  Definition \ref{definition1H} or Definition \ref{definition2H}.   
Let $L_c,L_f,\beta, \beta'$  be positive constants such that  $\beta,\beta'\in (0,1]$ and for every $x,y\in \mathbb{H}^1,$
 \begin{equation*}
 |c(x)-c(y)|\leq L_c|x-y|^{\beta},\quad   |f(x)-f(y)|\leq L_c|x-y|^{\beta'}.
 \end{equation*}
 If $c>0$ for every $x\in \mathbb{H}^1$  
 and 
 $$
 \inf_{x\in  B_R(P)}c(x):=c_0>0, 
 $$
 then
 there exist $0<\alpha:=\alpha(c_0,p, L_c,L_f,\Lambda)\in (0,1],$ $\alpha\leq \min\{\beta,\beta'\},$ and $L:=L(c_0,P, L_c,L_f,\Lambda)>0$ such that for every $x,y\in\mathbb{H}^1$
 $$
 |u(x)-u(y)|\leq L|x-y|^\alpha,
 $$ 
 that is $u\in C^{0,\alpha}(\mathbb{H}^1).$
\end{theorem}

 We conclude this introduction remarking that in \cite{Ishii1} Ishii proved that viscosity solutions of linear smooth second order elliptic operators, even possibly degenerate elliptic, have the same regularity of the functions $c,f\in C^{0,\beta}(\mathbb{R}^n), $ for every $\beta\in (0,1],$ representing respectively the zero order coefficient of the equation and the non-homogeneous term. We point out that in \cite{Ishii1} the case of a linear and complete operator with smooth coefficients has been treated. 
 See also the very interesting improvement obtained in \cite{JK}.
 Nevertheless the case of operators with smooth but unbounded coefficients is not completely discussed.
 
 In our paper we prove a $C^{0,\alpha}$  regularity result for uniformly continuous viscosity solutions  of degenerate equation as defined in Definition \ref{definition1H} and in Definition \ref{definition2H} , that is
 $$
 F(D^{2}u(x),x)-c(x)u(x)=f(x),
 $$
 where $F$ is homogeneous of degree one, $c\in C^{0,\beta}(\mathbb{H}^1), $  $f\in C^{0,\beta'}(\mathbb{H}^1), $ and, above all, $c(x)\geq c_0>0$ for every $x\in \mathbb{H}^1.$ 
 Thus, even using this approach, it is still open the case in which $\inf_{x\in \mathbb{H}^1}{c}=0$ and the case when $u$ is not uniformly continuous. For example even in the linear case, let say $F(D^{2}u(x))=\Delta_{\mathbb{H}^1}u(x),$ Ishii technique  seems does not work. 
 
Nevertheless, it is well known that,  in this last case for the sub-Laplacian in the Heisenberg group, a stronger result can be proved following a variational approach, see \cite{Hormander}. Indeed $\Delta_{\mathbb{H}^1}$ is a hypoelliptic operator. 
Anyhow we remark that our result applies also to viscosity solutions of equations that do not belong to the classes already discussed  in \cite{Hormander}, \cite{Ishii1}  or \cite{Ishii_Lions} and seems to give a new result even in the linear case.

Regularity of viscosity solutions is a subject that attracts the interests of many researchers. Thus we like to point out the following less recent and recent results about some properties of the solutions of nonlinear equations in the elliptic degenerate case:  \cite{LWang}, \cite{MGM}, \cite{MM}, \cite{BiGaIshi}, \cite{AGT} and, concerning the evolutive framework, \cite{LiuManfrediZhou}.

The paper is organized as follows, in Section \ref{preliminaryCarnot} we introduce main notation and the basic definitions in the Heisenberg group. Section \ref{preliminary}
is devoted to the regularity proof of the result for operator defined in the intrinsic way, see Definition \ref{definition2H}, namely without using matrix $P$. In Section \ref{nonintrinsica} we face the case of operators defined as in Definition \ref{definition1H}, that is using matrix $P$.
 
\section{Some preliminaries} \label{preliminaryCarnot}
In this section, in order to fix the notation we fix some basic facts about of the simplest non-trivial case of stratified Carnot group, the Heisenberg group $\mathbb{H}^1.$ In addition we recall the notion of viscosity solutions in this framework.  

\subsection{The Heisenberg group}
Given a group $\mathbb{G}\equiv\mathbb{R}^n$ endowed with the inner non-commutative group law $\cdot$ and algebra $\mathfrak{g}\equiv \mathbb{R}^n,$ we say that $\mathbb{G}$ is a stratified Carnot group if there exist $\{\mathfrak{g}_i\}_{1\leq i\leq m}$ vector spaces of $\mathfrak{g}$
such that
$$
\mathfrak{g}=\bigoplus_{k=1}^{m}\mathfrak{g}_k,
$$
and for $k=1,\dots,m-1$
$$
[\mathfrak{g}_1,\mathfrak{g}_k]=\mathfrak{g}_{k+1},
$$
where $[X,Y]$ denotes the commutator of two vector fields of the algebra $\mathfrak{g}.$

The simplest case is given by the Heisenberg group $\mathbb{H}^1.$ Indeed, in this case $\mathbb{H}^1\equiv \mathbb{R}^3,$ $\mathfrak{h}^1\equiv\mathbb{R}^3$ and for every
$x=(x_1,x_2,x_3), y=(y_1,y_2,y_3)\in \mathbb{H}^1,$ it is defined the non-commutative inner law
$$
x\cdot y=(x_1+y_1,x_2+y_2,x_3+y_3+2(y_1x_2-y_2x_1))
$$
and for every  $x\in \mathbb{H}^1,$ $-x=(-x_1,-x_2,-x_3)$ is the opposite of $x.$

The algebra $\mathfrak{h}=\mathfrak{h}_1\bigoplus \mathfrak{h}_2,$ where
$$
\mathfrak{h}_1=\mbox{span}\{X,Y\},\quad \mathfrak{h}_2=\mbox{span}\{T\},
$$
 $X=\frac{\partial}{\partial x_1}+2x_2\frac{\partial}{\partial x_3},$ $Y=\frac{\partial}{\partial x_2}-2x_1\frac{\partial}{\partial x_3}$ and $T=\frac{\partial}{\partial x_3}.$ In particular
 $$
 [X,Y]=-4T
 $$
 and
 $$
 [\mathfrak{h}_1,\mathfrak{h}_1]=\mathfrak{h}_2.
 $$
 The vector fields $X$ and $Y$ are identified respectively with the vectors $(1,0,2x_2)$ and $(0,1,-2x_1)$ so that we can write $X(x)=(1,0,2x_2)$ and $Y=(0,1,-2x_1).$ We remark, for instance, that 
 taking the solution of the Cauchy problem 
 \begin{equation}
 \left\{
 \begin{array}{l}
 \phi'=X(\phi)\\
 \phi(0)=x,
 \end{array}
 \right.
 \end{equation}
 then for every function $u$ sufficiently smooth we get $(u\circ\phi)'(0)=\frac{\partial u}{\partial x_1}(x)+2x_2\frac{\partial u}{\partial x_3}(x)=Xu(x)$ and an analogous computation may be done for $Y.$ We denote by 
 $$\nabla_{\mathbb{H}^1}u(x)=Xu(x)X(x)+Yu(x)Y(x)=(Xu(x),Yu(x))$$
 the intrinsic gradient.
 It is also possible to define a second order object analogous to the Hessian matrix, even if the structure of $\mathbb{H}^1$ is not commutative. Indeed, we define the simmetrized horizontal Hessian matrix of $u$ at $x$ as  follows:
 
 \begin{equation}
D^{2,*}_{\mathbb{H}^1}u(x)=\left[
\begin{array}{lr}
X^2u(x),&\frac{(XY+YX)u(x)}{2}\\
\frac{(XY+YX)u(x)}{2},&Y^2u(x)
\end{array}
\right].
\end{equation}
 It is important to remark the differences with respect to the classical $\nabla u$ and the classical Hessian matrix $D^2u(x)$ in $\mathbb{R}^3$ that is of course a $3\times3$ matrix. Indeed, $(Xu(x),Yu(x))\in \mathbb{R}^2,$ while $\nabla u(x)\mathbb{R}^3$ and  $D^{2,*}_{\mathbb{H}^1}u(x)$ is a $2\times2$ matrix instead to be a $3\times3$ matrix.
 Nevertheless the following result is true.
 
 \begin{lemma}
 Let $\Omega\subseteq\mathbb{R}^3$ be an open set. If $u\in C^2(\Omega)$ then:
\begin{equation}\begin{split}
&D^{2*}u(x)=\left[
\begin{array}{lr}
X^2u(x),&\frac{XYu(x)+YXu(x)}{2}\\
\frac{XYu(x)+YXu(x)}{2},&Y^2u(x)
\end{array}
\right]\\
&=\left[
\begin{array}{lr}
\langle D^2u(x) X(x),X(x)\rangle,&\langle D^2u(x) X(x),Y(x)\rangle\\
\langle D^2u(x) X(x),Y(x)\rangle,&\langle D^2u(x) Y(x),Y(x)\rangle
\end{array}
\right]
\end{split}
\end{equation}

Moreover, for every $\alpha, \beta\in \mathbb{R}$ and for every $x\in \Omega\subseteq\mathbb{H}^1$ if $u\in C^2(\Omega),$ then
$$
\langle D^2u(\alpha X+\beta Y)^T,(\alpha X+\beta Y)\rangle=\langle D^{2*}_{\mathbb{H}^1}u(\alpha,\beta)^T,(\alpha,\beta)\rangle
$$
\end{lemma}
\begin{proof} By straightforward calculation we obtain:
\begin{equation*}\begin{split}
&\langle D^2u(\alpha X+\beta Y)^T,(\alpha X+\beta Y)\rangle=\alpha^2\langle D^2uX^T,X\rangle+2\alpha\beta\langle D^2u X^T, Y\rangle\\
&+ \beta^2\langle D^2u Y^T,Y\rangle\\
&=\alpha^2\langle (\frac{\partial^2 u}{\partial x_1^2}+2x_2 \frac{\partial^2 u}{\partial x_3\partial x_1},\frac{\partial^2 u}{\partial x_1\partial x_2}-2x_1 \frac{\partial^2 u}{\partial x_3\partial x_2},\frac{\partial^2 u}{\partial x_1\partial x_3}+2x_2 \frac{\partial^2 u}{\partial x_3^2}),X\rangle\\
&+2\alpha\beta\langle (\frac{\partial^2 u}{\partial x_1^2}+2x_2 \frac{\partial^2 u}{\partial x_3\partial x_1},\frac{\partial^2 u}{\partial x_1\partial x_2}+2x_2 \frac{\partial^2 u}{\partial x_3\partial x_2},\frac{\partial^2 u}{\partial x_1\partial x_3}+2x_2 \frac{\partial^2 u}{\partial x_3^2}), Y\rangle\\
&+\beta^2\langle (\frac{\partial^2 u}{\partial x_2\partial_1}-2x_1 \frac{\partial^2 u}{\partial x_3\partial x_1},\frac{\partial^2 u}{\partial x_2^2}-2x_1 \frac{\partial^2 u}{\partial x_3\partial x_2},\frac{\partial^2 u}{\partial x_2\partial x_3}-2x_1 \frac{\partial^2 u}{\partial x_3^2}), Y\rangle\\
&=\alpha^2\left(\frac{\partial^2 u}{\partial x_1^2}+4x_2\frac{\partial^2 u}{\partial x_1\partial x_3}+4x_2^2 \frac{\partial^2 u}{\partial x_3^2}\right)\\
&+2\alpha\beta\left(\frac{\partial^2 u}{\partial x_1\partial x_2}+2x_2 \frac{\partial^2 u}{\partial x_3\partial x_2}-2x_1 \frac{\partial^2 u}{\partial x_3\partial x_1}-4x_1x_2 \frac{\partial^2 u}{\partial x_3^2}\right)\\
&+\beta^2\left(\frac{\partial^2 u}{\partial x_2^2}-4x_1\frac{\partial^2 u}{\partial x_2\partial x_3}+4x_1^2 \frac{\partial^2 u}{\partial x_3^2}\right).
\end{split}
\end{equation*}
On the other hand

\begin{equation}
\begin{split}
&\langle D^{2*}_{\mathbb{H}^1}u(\alpha,\beta)^T,(\alpha,\beta)\rangle=\alpha^2X^2u+2\frac{\alpha\beta}{2} (XY+YX)u+\beta^2YYu\\
&=\alpha^2\left(\frac{\partial^2u}{\partial x_1^2 }+4x_2\frac{\partial^2u}{\partial x_1\partial x_3 }+4x_2^2\frac{\partial^2u}{\partial x_3^2 }\right)\\
&+\alpha\beta\left(2\frac{\partial^2u}{\partial x_1\partial x_2 }+4x_2\frac{\partial^2u}{\partial x_3\partial x_2}-4x_1\frac{\partial^2u}{\partial x_3\partial x_1 }-8x_2x_1\frac{\partial^2u}{\partial x_3^2 }\right)\\
&+\beta^2\left(\frac{\partial^2u}{\partial x_2^2 }-4x_1\frac{\partial^2u}{\partial x_2\partial x_3 }+4x_1^2\frac{\partial^2u}{\partial x_3^2 }\right).
\end{split}
\end{equation}
\end{proof}

\begin{lemma}
Let $A,$ $B$ be $n\times n$ symmetric matrices. Assume that $\phi\in C^2(\Omega\times \Omega),$ where $\Omega\subseteq\mathbb{H}^1.$
If
\begin{equation}
\left[
\begin{array}{lr}
A,&0\\
0,&-B
\end{array}
\right]\leq D^2\phi(x,y)+\frac{1}{\mu}(D^2\phi(x,y))^2
\end{equation}
then
\begin{equation}
\begin{split}
&\langle A(\alpha_1X+\beta_1Y),(\alpha_1X+\beta_1Y)\rangle-\langle B(\alpha_1X+\beta_1Y),(\alpha_1X+\beta_1Y)\rangle\\
& \leq \langle D^2\phi(x,y)
\left[
\begin{array}{l}
\alpha_1 X+\beta_1 Y,\\
\alpha_2 X+\beta_2 Y
\end{array}
\right],[\alpha_1 X+\beta_1 Y,\alpha_2 X+\beta_2 Y]\rangle
\\
&+\frac{1}{\mu}\langle(D^2\phi(x,y))^2\left[
\begin{array}{l}
\alpha_1 X+\beta_1 Y\\
\alpha_2 X+\beta_2 Y,
\end{array}
\right],[\alpha_1 X+\beta_1 Y,\alpha_2 X+\beta_2 Y]\rangle.
\end{split}
\end{equation}

\end{lemma}
\begin{proof} It is a simple computation. Indeed:
\begin{equation}
\begin{split}
&\langle\left[
\begin{array}{lr}
A,&0\\
0,&-B
\end{array}
\right] \left[
\begin{array}{l}
\alpha_1 X+\beta_1 Y\\
\alpha_2 X+\beta_2 Y
\end{array}
\right],[\alpha_1 X+\beta_1 Y,\alpha_2 X+\beta_2 Y]\rangle\\
& \leq \langle D^2\phi(x,y)
\left[
\begin{array}{l}
\alpha_1 X+\beta_1 Y,\\
\alpha_2 X+\beta_2 Y
\end{array}
\right],[\alpha_1 X+\beta_1 Y,\alpha_2 X+\beta_2 Y]\rangle
\\
&+\frac{1}{\mu}\langle(D^2\phi(x,y))^2\left[
\begin{array}{l}
\alpha_1 X+\beta_1 Y\\
\alpha_2 X+\beta_2 Y,
\end{array}
\right],[\alpha_1 X+\beta_1 Y,\alpha_2 X+\beta_2 Y]\rangle.
\end{split}
\end{equation}
Hence
\begin{equation}
\begin{split}
&\langle A(\alpha_1X+\beta_1Y),(\alpha_1X+\beta_1Y)\rangle-\langle B(\alpha_1X+\beta_1Y),(\alpha_1X+\beta_1Y)\rangle\\
& \leq \langle D^2\phi(x,y)
\left[
\begin{array}{l}
\alpha_1 X+\beta_1 Y,\\
\alpha_2 X+\beta_2 Y
\end{array}
\right],[\alpha_1 X+\beta_1 Y,\alpha_2 X+\beta_2 Y]\rangle
\\
&+\frac{1}{\mu}\langle(D^2\phi(x,y))^2\left[
\begin{array}{l}
\alpha_1 X+\beta_1 Y\\
\alpha_2 X+\beta_2 Y,
\end{array}
\right],[\alpha_1 X+\beta_1 Y,\alpha_2 X+\beta_2 Y]\rangle.
\end{split}
\end{equation}
\end{proof}

In the Heisenberg group a homogeneous semigroup of dilation is defined. Namely, for every $\lambda>0$ and for every $x\in \mathbb{H}^1$
$$
\delta_{\lambda}(x)=(\lambda x_1,\lambda x_2,\lambda^2 x_3).
$$

Moreover considering $u$ sufficiently smooth, we get: $Xu(\delta_\lambda x)=\lambda (Xu)(\delta_\lambda x),$ $Yu(\delta_\lambda x)=\lambda (Yu)(\delta_\lambda x)$ and
$$
\Delta_{\mathbb{H}^1}u(\delta_\lambda x)=\lambda^2(\Delta_{\mathbb{H}^1}u)(\delta_{\lambda}x).
$$
\subsection{Viscosity solutions}
\begin{definition}
  We recall the definition of viscosity solution coherently with the statement given in \cite{userguide}.
  
  We say that   $u\in C(\Omega)$ is a subsolution of  (\ref{defvisco}) if for every $\phi\in C^2$ and for every $x_0\in \Omega$ if $u-\phi$ realizes a maximum at $x_0$ in a open neighborhood $U_{x_0}$ of $x_0$ then
 $$F(D^{2}\phi(x_0),x_0)-c(x_0)\phi(x_0)\geq f(x_0).$$ 
 
 Analogously we shall say that   $u\in C(\Omega)$ is a supersolution of  (\ref{defvisco}) if for every $\phi\in C^2$ and for every $x\in \Omega$ if $u-\phi$ realizes a minimum at $x_0$ in a open neighborhood $U_{x_0}$ of $x_0$ then
$$F(D^{2}\phi (x_0),x_0)-c(x_0)u(x_0)\leq f(x_0).$$

If $u\in C(\Omega)$ is both a subsolution and a supersolution of (\ref{defvisco}), then $u$ is a viscosity solution of the equation (\ref{defvisco}).
\end{definition}

Here we wish to recall also the intrinsic definition of viscosity solution concerning subelliptic semi-jets, see \cite{Manfredi}, \cite{Bieske}. Nevertheless we also want to stress that using the usual notion of viscosity solution given by Ishii and Lions our result is true.

The second order subelliptic superjet of $u$ at $p_0$ is defined as follows
\begin{definition}

Let $u$ be an upper-semicontinuous real function in $\mathbb{H}^1.$ 
\begin{equation}\begin{split}
J^{2,+}_{\mathbb{H}^1}(u,p_0)&=\big\{(\eta,\mathcal X)\in \mathbb{R}^3\times\mathcal{S}(\mathbb{R}):\quad\\
&u(p)\leq u(p_o)+\langle\eta,p_o^{-1}\cdot p\rangle+\frac{1}{2}\langle\mathcal{X}(p_o^{-1}\cdot p)',(p_o^{-1}\cdot p)'\rangle+o(|p_o^{-1}\cdot p|^2)\big\}.
\end{split}
\end{equation}
\end{definition}
An analogous definition can be given for second order subelliptic subjet of $u$ at $p_0.$
\begin{definition}
Let $u$ be an lower-semicontinuous real function in $\mathbb{H}^1.$
\begin{equation}\begin{split}
J^{2,-}_{\mathbb{H}^1}(u,p_0)&=\big\{(\eta,\mathcal X)\in \mathbb{R}^3\times\mathcal{S}(\mathbb{R}):\quad\\
&u(p)\geq u(p_o)+\langle\eta,p_o^{-1}\cdot p\rangle+\frac{1}{2}\langle\mathcal{X}(p_o^{-1}\cdot p)',(p_o^{-1}\cdot p)'\rangle+o(|p_o^{-1}\cdot p|^2)\big\}.
\end{split}
\end{equation}
\end{definition}
In case the function $u$ is smooth, it results
\begin{equation}\begin{split}\label{fundsub}
u(p)&=u(p_o)+\langle\nabla u(p_0),p_o^{-1}\cdot p\rangle+\frac{1}{2}\langle D^{2,*}_{\mathbb{H}^1}u(p_0)(p_o^{-1}\cdot p)',(p_o^{-1}\cdot p)'\rangle+o(|p_o^{-1}\cdot p|^2)\\
&=u(p_o)+\frac{\partial u}{\partial x}(p_0)(p_1-p_{01})+\frac{\partial u}{\partial y}(p_0)(p_2-p_{02})+\frac{\partial u}{\partial t}(p_3-p_{03}+2(p_2p_{01}-p_1p_{02}))\\
&+\frac{1}{2}\langle D^{2,*}_{\mathbb{H}^1}u(p_0)(p_o^{-1}\cdot p)',(p_o^{-1}\cdot p)'\rangle+o(|p_o^{-1}\cdot p|^2)\\
&=u(p_o)+Xu(p_0)(p_1-p_{01})+Yu(p_0)(p_2-p_{02})+\frac{\partial u}{\partial t}(p_0)(p_3-p_{03})\\
&+\frac{1}{2}\langle D^{2,*}_{\mathbb{H}^1}u(p_0)(p_o^{-1}\cdot p)',(p_o^{-1}\cdot p)'\rangle+o(|p_o^{-1}\cdot p|^2).
\end{split}
\end{equation}

\begin{remark}
The Lemma 3.4 in 
\cite{Bieske}
 can be re-formulated as follows. We denote with $\bar{J}^{2,+}u(p_0)$ and $\bar{J}^{2,-}u(p_0),$ as usual, the classical super-jet and sub-jet  of $u$ in $p_0.$
If $\eta\in \mathbb{R}^{2n+1}$ is a vector and $A$ is a symmetric $2n+1\times 2n+1$ matrix. If $(\eta, A)\in \bar{J}^{2,+}u(p_0),$ then $(\tilde{\eta}, \tilde{A})\in \bar{J}_{\mathbb{H}^1}^{2,+}(p_0)$  where
$$
\tilde{\eta}_1=\langle\eta,X\rangle ,\quad \tilde{\eta}_2=\langle\eta,Y\rangle,\quad \tilde{\eta}_3=\langle\eta,[X,Y]\rangle
$$
and
\begin{equation}
\tilde{A}_{p_0}=\left[
\begin{array}{lr}
\langle A X(p_0),X(p_0)\rangle,&\langle A X(p_0),Y(p_0)\rangle\\
\langle A X(p_0),Y(p_0)\rangle,&\langle A Y(p_0),Y(p_0)\rangle
\end{array}
\right].
\end{equation}
\end{remark}

\subsection{Theorem of the sums}
  This result is described with different names, for instance, in the papers \cite{Ishii1} and \cite{Ishii_Lions} see also  \cite{LuiroParviainen}. In the following part 
  ${J}^{2,+}u(\hat{x})$ and ${J}^{2,-}u(\hat{x})$ denote respectively the classical superjet and subjet of $u$ at the point $\hat{x}\in \Omega$ for the function $u\in C(\Omega).$
  \begin{theorem}
  Let $u\in USC(\bar{\Omega})$ and $v\in LSC(\bar{\Omega}).$ For $\phi\in C^{2}(\mathbb{R}^n\times\mathbb{R}^n).$ If there exists $(\hat{x},\hat{y})\in\Omega\time\Omega$ such that
  \begin{equation}
  \begin{split}
  u(\hat{x})-v(\hat{y})-\phi(\hat{x},\hat{y})=\max_{(x,y)\in\overline{\Omega}\times\overline{\Omega}}\left(u(x-v(y)-\phi(x,y))\right),
  \end{split}
  \end{equation}
  then for each $\mu>0,$ there are $A=A(\mu)$ and $B=B(\mu)$ such that
  $$
  (D_x\phi(\hat{x},\hat{y}),A)\in\overline{J}^{2,+}u(\hat{x}),\quad (-D_y\phi(\hat{x},\hat{y}),B)\in\overline{J}^{2,+}u(\hat{y})
  $$
  and
  \begin{equation}
  \begin{split}
  -\left(\mu+||D^2\phi(\hat{x},\hat{y})||\right)&
  \left[\begin{array}{cc}
  I,&0\\
  0,&I
  \end{array}
  \right]
  \leq
  \left[\begin{array}{cc}
  A,&0\\
  0,&-B
  \end{array}
  \right]\\
  &\leq D^2\phi(\hat{x},\hat{y})+\frac{1}{\mu}(D^2\phi(\hat{x},\hat{y}))^2,
  \end{split}
  \end{equation}
  where
  $$
  D^2\phi(\hat{x},\hat{y})= \left[\begin{array}{cc}
  D^2_{xx}\phi(\hat{x},\hat{y}),&D^2_{yx}\phi(\hat{x},\hat{y})\\
  D^2_{xy}\phi(\hat{x},\hat{y}),&D^2_{yy}\phi(\hat{x},\hat{y})
  \end{array}
  \right]
  $$
  and $||D||$ is the norm given by the maximum, in absolute value, of the eigenvalues of the symmetric matrix $D.$
  \end{theorem}
  \begin{lemma}\label{teochar}
  Let $\phi(x,y)=|x-y|^{\alpha}.$ If $x\not=y$ then
  \begin{equation}
  \begin{split}
 D^2\phi(x,y)= \left[\begin{array}{cc}
 M,&-M\\
  -M,&M
  \end{array}
  \right],
  \end{split}
  \end{equation}
where 
\begin{equation}
  \begin{split}
  M=\alpha |x-y|^{\alpha-2}\left((\alpha-2)\frac{x-y}{|x-y|}\otimes\frac{x-y}{|x-y|}+I\right)
   \end{split}
  \end{equation}
  and
   \begin{equation}
  \begin{split}
  (D^2\phi(x,y))^2=2\left[\begin{array}{cc}
 M^2,&-M^2\\
  -M^2,&M^2
  \end{array}
  \right],
  \end{split}
  \end{equation}
 where
  \begin{equation}
  \begin{split}
  M^2=\alpha^2 |x-y|^{2(\alpha-2)}\left(\alpha(\alpha-2)\frac{x-y}{|x-y|}\otimes\frac{x-y}{|x-y|}+I\right).
   \end{split}
  \end{equation}

  \end{lemma}
 \section{Regularity result}\label{preliminary}

We schedule the proof of the main theorem arranging some intermediate steps.

Let $\alpha \in (0,1].$ We consider 
$$w(x,y)=u(x)- u(y)- L |x-y|^{\alpha} - \delta |x|^2-\epsilon.$$
where $u$ is a continuous bounded function defined in all of $\mathbb{R}^n.$
For simplicity, let us denote by $\varphi(x,y):= L |x-y|^{\alpha}.$\\
Assume for the moment that we have satisfied the hypothesis
requested to apply the Theorem of sums. This means that, still denoting with $(x,y)$ the point that realizes the maximum  there exist
two symmetric matrices $A,B \in \mathcal{S}^{3}$ such that
$$\left( D_x \varphi(x,y) + 2\delta (x-z); A+2\delta I\right) \in \bar{J}^{2,+}u(x) $$
and
$$ \left(-D_y \varphi(x,y); B\right) \in \bar{J}^{2,-}u(y),$$
with
\begin{equation}\label{eq:Sums}
\left( \begin{array}{cc}
             A & 0\\
						 0 & -B 
         \end{array}\right) \leq D^{2}\varphi + \dfrac{1}{\mu}(D^2 \varphi)^{2} 
				=: \left( \begin{array}{cc}
				            M & -M\\
										-M & M 
										\end{array}\right) + \dfrac{2}{\mu} \left( \begin{array}{cc}
										                                               M^{2} & -M^2\\
																																	 -M^2 & M^2 
																																	\end{array}\right),
																																	\end{equation}
where 

\begin{equation}\label{emme1}
M=L\alpha |x-y|^{\alpha-2}\left((\alpha-2)\frac{x-y}{|x-y|}\otimes\frac{x-y}{|x-y|}+I\right)
\end{equation}
and, by straightforward calculation, 
\begin{equation}\label{emme2}
M^2=L^2\alpha^2 |x-y|^{2(\alpha-2)}\left(\alpha(\alpha-2)\frac{x-y}{|x-y|}\otimes\frac{x-y}{|x-y|}+I\right).
\end{equation}
It is worth while to remark that the smallest eigenvalue of
$$
(\alpha-2)\frac{x-y}{|x-y|}\otimes\frac{x-y}{|x-y|}+I
$$
is
$$
\alpha-2+1=\alpha-1
$$
associated with the eigenvector $\frac{x-y}{|x-y|},$ while the largest is $1.$

\begin{remark}\label{osservazione1}
In the sequel with the matrix $N$ we define:
$$
N=M+\frac{2}{\mu}M^2.
$$
In particular recalling (\ref{emme1}) and (\ref{emme2}) we get
$$
|N|\leq L\alpha |x-y|^{\alpha-2}+\frac{2}{\mu}L^2\alpha^2 |x-y|^{2(\alpha-2)}.
$$
\end{remark}
 
\begin{lemma}\label{52}
Let $A,B$ and $N$ symmetric matrices such that:
\begin{equation}
\left[
\begin{array}{lr}
A,&0\\
0,&-B
\end{array}
\right]\leq \left[
\begin{array}{lr}
N,&-N\\
-N,&N
\end{array}
\right],
\end{equation}
then, in particular, for every $[\xi,\eta]\in \mathbb{R}^{2n}$
\begin{equation}
[\xi,\eta]\left[
\begin{array}{lr}
A,&0\\
0,&-B
\end{array}
\right]\left[\begin{array}{l}\xi\\
\eta\end{array}\right]\leq  [\xi,\eta]\left[
\begin{array}{lr}
N,&-N\\
-N,&N
\end{array}
\right]\left[\begin{array}{l}\xi\\
\eta\end{array}\right]
\end{equation}
and
$$
\langle A\xi,\xi\rangle-\langle B\eta,\eta\rangle\leq  \langle N(\xi-\eta),(\xi-\eta)\rangle
$$

In addition we get that, for every $a,b\geq 0$ and for every $\xi,\eta\in \mathbb{R}^n$
\begin{equation}
\begin{split}
&a\langle A\xi_1,\xi_1\rangle+b\langle A\xi_2,\xi_2\rangle-a\langle B\eta_1,\eta_1\rangle-b\langle B\eta_2,\eta_2\rangle\\
&\leq  \left(a\langle N(\xi_1-\eta_1),(\xi_1-\eta_1)\rangle+b\langle N(\xi_2-\eta_2),(\xi_2-\eta_2)\rangle\right).
\end{split}
\end{equation}
\end{lemma}
\begin{proof}
The result follows by straightforward calculation.
\end{proof}

\begin{remark}
If $A_x-B_y$ has all the eigenvalues negative, 
\begin{equation}
\begin{split}
&\langle A_xX(x),X(x)\rangle-\langle B_yX(y),X(y)\rangle\\
&\leq \langle M(X(x)-X(y),X(x)-X(y))+\frac{1}{\mu}\langle M^2(X(x)-X(y),X(x)-X(y))
\end{split}
\end{equation}
and
\begin{equation}
\begin{split}
&\langle A_xY(x),Y(x)\rangle-\langle B_yY(y),Y(y)\rangle\\
&\leq \langle M(Y(x)-Y(y),Y(x)-Y(y))+\frac{1}{\mu}\langle M^2(Y(x)-Y(y),Y(x)-Y(y)).
\end{split}
\end{equation}
Then
\begin{equation}
\begin{split}
&\langle A_xX(x),X(x)\rangle-\langle B_yX(y),X(y)\rangle+\langle A_xY(x),Y(x)\rangle-\langle B_yY(y),Y(y)\rangle\\
&\leq 2\langle M(Y(x)-Y(y),Y(x)-Y(y))+\frac{2}{\mu}\langle M^2(Y(x)-Y(y),Y(x)-Y(y))
\end{split}
\end{equation}
That is
\begin{equation}\label{35rem}
\begin{split}
&\langle A_xX(x),X(x)\rangle+\langle A_xY(x),Y(x)\rangle-(\langle B_yX(y),X(y)\rangle+\langle B_yY(y),Y(y)\rangle)\\
&\leq \langle M(Y(x)-Y(y),Y(x)-Y(y))+\frac{1}{\mu}\langle M^2(Y(x)-Y(y),Y(x)-Y(y))\\
+&\langle M(X(x)-X(y),X(x)-X(y))+\frac{1}{\mu}\langle M^2(X(x)-X(y),X(x)-X(y))
\end{split}
\end{equation}
Since, $X(x)-X(y)=(0,0,2(y_1-x_1))$ and $Y(x)-Y(y)=(0,0,2(x_2-y_2)),$ we get
\begin{equation}\label{36remx}
\begin{split}
&\langle M(X(x)-X(y),X(x)-X(y))+\frac{1}{\mu}\langle M^2(X(x)-X(y),X(x)-X(y))\\
&=4(x_1-y_1)^2\langle Me_3,e_3\rangle
\end{split}
\end{equation}
and
\begin{equation}\label{36remy}
\begin{split}
&\langle M(Y(x)-Y(y),Y(x)-Y(y))+\frac{1}{\mu}\langle M^2(Y(x)-Y(y),Y(x)-Y(y))\\
&=4(x_2-y_2)^2\langle Me_3,e_3\rangle.
\end{split}
\end{equation}
As a consequence
\begin{equation}\label{35remb}
\begin{split}
&\langle A_xX(x),X(x)\rangle+\langle A_xY(x),Y(x)\rangle-(\langle B_yX(y),X(y)\rangle+\langle B_yY(y),Y(y)\rangle)\\
&\leq 4((x_1-y_1)^2+(x_2-y_2)^2)\langle Me_3,e_3\rangle +\frac{4}{\mu}((x_1-y_1)^2+(x_2-y_2)^2)\langle M^2e_3,e_3\rangle
\end{split}
\end{equation}
\end{remark}

\begin{corollary}\label{traceresult}
Let $A,B$ and $N$ symmetric matrices such that:
\begin{equation}
\left[
\begin{array}{lr}
A,&0\\
0,&-B
\end{array}
\right]\leq \left[
\begin{array}{lr}
N,&-N\\
-N,&N
\end{array}
\right],
\end{equation}
then if  $\xi_1=X(x),$ $\eta_1=X(y),$ $\xi_2=Y(x)$ and $\eta_2=Y(y),$ we get that for every $a,b\geq 0$
\begin{equation}
\begin{split}
&a\langle A X(x),X(x)\rangle+b\langle AY(x),Y(x)\rangle-a\langle A X(y),X(y)\rangle-b\langle AY(y),Y(y)\rangle\\
&\leq \left(a (x_2-y_2)^2+b(x_1-y_1)^2\right) n_{33},
\end{split}
\end{equation}
where $n_{33}=\langle Ne_3,e_3\rangle.$
In particular, if $a=1=b$
\begin{equation}
\begin{split}
\mbox{Tr}(\tilde{A})-\mbox{Tr}(\tilde{B})\leq \left(a (x_2-y_2)^2+b(x_1-y_1)^2\right) n_{33},
\end{split}
\end{equation}
\end{corollary}
\begin{proof}
Keeping in mind Lemma \ref{52} and choosing
$\xi_1-\eta_1=2(0,0,x_2-y_2)$ and $\xi_2-\eta_2=-2(0,0,x_1-y_1),$ then the result immediately follows.
\end{proof}

In this section we prove our main result. 
For sake of simplicity we denote by $||u||_\infty=||u||_{L^{\infty}(\mathbb{R}^n)}.$

\begin{proof}[Proof of Theorem \ref{mainregularityresult}]

Let $u$ be a viscosity solution of
$$
F(D^{2*}u(x))-c(x)u(x)=f(x), \quad x\in \mathbb{H}^1
$$
uniformly continuous.

Let us define  for $\alpha\in (0,1]$ the function
$$
\psi(x,y):=u(x)-u(y)-L|x-y|^\alpha-\delta |x|^2-\epsilon,
$$
where $L$ and $\delta,$ $\epsilon$ are positive constants.

Let us define

$$
\theta:=\sup_{\mathbb{H}^1\times \mathbb{H}^1}\psi(x,y)
$$

We claim that there exist $L_0>0,$  and $\alpha_0=\alpha_0(||u||_{\infty},\Lambda,\lambda,c_0)\in (0,1]$ such that for every $\delta>0$ and for every $\epsilon>0,$ then $\theta\leq 0.$

 In this case
$$
\psi(x,y)\leq u(\hat{x})-u(\hat{y})-L_0|\hat{x}-\hat{y}|^{\alpha_0}-\delta |\hat{x}|^2-\epsilon\leq 0
$$
and letting $\delta\to 0$ and  $\epsilon\to 0$ we get the thesis, that is for every $x,y\in \mathbb{H}^1\equiv\mathbb{R}^3$

$$
u(x)-u(y)-L_0|x-y|^{\alpha_0}\leq 0.
$$

We argue by contradiction.  Let us suppose that there exist $\delta_0$ and $\epsilon_0$  such that for every $\delta<\delta_0$ and for every $\epsilon<\epsilon_0$
$$
0<\theta=\sup_{\mathbb{R}^n\times\mathbb{R}^n}\{u(x)-u(y)-L|x-y|^\alpha-\delta |x|^2-\epsilon\}.
$$

 Then for every $0<\delta<\delta_0,$ and for every $0<\epsilon<\epsilon_0$ fixed there exists a sequence $\{(x_j,y_j)\}_{j\in \mathbb{N}}$ such that
$$
\lim_{j\to \infty}(u(x_j)-u(y_j)-L|x_j-y_j|^\alpha-\delta |x_j|^2-\epsilon)=\theta,
$$

On the other hand $\theta >0,$ so that there exists $\bar{j}\in \mathbb{N}$ such that for every $j>\bar{j},$ $u(x_j)-u(y_j)-L|x_j-y_j|^\alpha-\delta |x_j|^2-\epsilon>\frac{\theta}{2}>0.$ Then for every $j>\bar{j}$
$$
\infty>2||u||_\infty\geq u(x_j)-u(y_j)>L|x_j-y_j|^\alpha+\delta |x_j|^2+\epsilon.
$$
Thus
$$
\delta |x_j|^2<2||u||_\infty, \quad \mbox{and}\quad |x_j-y_j|^\alpha<\frac{2||u||_\infty}{L}.
$$
By compactness, possibly extracting a subsequence, we get that there exists $(\hat{x},\hat{y})\in \mathbb{R}^n\times \mathbb{R}^n$ such that
$$
\lim_{k\to\infty}x_{j_k}=\hat{x},
$$
$$
\lim_{k\to\infty}y_{j_k}=\hat{y}.
$$
It is clear that in principle, $(\hat{x},\hat{y}):=(\hat{x}(\delta,L,\epsilon),\hat{y}(\delta, L,\epsilon)).$
Possibly taking $L$  sufficiently large,  by the uniformly continuity of $u$ we would get a contradiction whenever $\left(\frac{2||u||_\infty}{L}\right)^{\frac{1}{\alpha}}<\eta(\epsilon),$ where $\eta$ is the parameter independent of $\hat{x}, \hat{y}\in \mathbb{H}^1$ associated with the uniformly continuity.  So that there exists 
 $0<\gamma$ such that
$$
\lim_{k\to\infty}|x_{j_k}-y_{j_k}|^\alpha=\gamma.
$$
Indeed, we get that $|\hat{x}-\hat{y}|^\alpha=\gamma>0,$ independently to $\delta.$ Thus, we can assume that $\gamma$ does not depend on $\delta.$

As a consequence, if $\theta>0$ then there exists $(\hat{x},\hat{y})\in \mathbb{H}^1$ and $\gamma>0$ such that
$$
\theta:=\sup_{\mathbb{H}^1\times \mathbb{H}^1}\psi(x,y)=\psi(\hat{x},\hat{y}),
$$
and there exists $\bar{\gamma}$ such that independently form $\delta$ such that:
$$
0<\bar{\gamma}< |\hat{x}-\hat{y}|^\alpha<\frac{2||u||_\infty}{L}.
$$

 Then
recalling that $u$ is a viscosity solution and keeping in mind the hypotheses on $c,$ and $f$ we get by Theorem of the sums:

\begin{equation}\label{mainine}
\begin{split}
& c_0(L|\hat{x}-\hat{y}|^\alpha+\delta |\hat{x}|^2)\\
&\leq c_0(u(\hat{x})-u(\hat{y}))\leq c(\hat{x})(u(\hat{x})-u(\hat{y}))\\
&=c(\hat{x})u(\hat{x})-c(\hat{y})u(\hat{y})+u(\hat{y})(c(\hat{y})-c(\hat{x}))\\
&\leq F(\tilde{A}+2\delta\tilde{I})-F(\tilde{B})+f(\hat{y})-f(\hat{x})+u(\hat{y})(c(\hat{y})-c(\hat{x})).
\end{split}
\end{equation}

We consider now two cases.

If $\tilde{A}+2\delta\tilde{I}\geq \tilde{B},$ then keeping in mind the hypothesis on $F,$ see Definition \ref{definition2H}, we get
\begin{equation}
\begin{split}
&c_0L|\hat{x}-\hat{y}|^\alpha+c_0\delta |\hat{x}|^2 \leq \Lambda \mbox{Tr}(\tilde{A}-\tilde{B})+2\delta(|X(\hat{x})|^2+|Y(\hat{x})|^2)\\
&+f(\hat{y})-f(\hat{x})+u(\hat{y})(c(\hat{y})-c(\hat{x})).
\end{split}
\end{equation}
Let us denote 
$$
R=|X(\hat{x})|^2+|Y(\hat{x})|^2= 2(1+2|\hat{x}'|^2).
$$
Moreover, by the Theorem of the sum and recalling previous Corollary \ref{traceresult}, we deduce that:
$$
\mbox{Tr}(\tilde{A}-\tilde{B})\leq  \left((\hat{x}_2-\hat{y}_2)^2+(\hat{x}_1-\hat{y}_1)^2\right) n_{33}.
$$
We recall also that keeping in mind the classical computation
\begin{equation*}
\begin{split}
n_{33}&= L\alpha|\hat{x}-\hat{y}|^{\alpha-2}\left((\alpha-2)\frac{(\hat{x}_3-\hat{y}_3)^2}{|\hat{x}-\hat{y}|^2}+1\right)\\
&+\frac{2}{\mu}L^2\alpha^2 |\hat{x}-\hat{y}|^{2(\alpha-2)}\left(\alpha(\alpha-2)\frac{(\hat{x}_3-\hat{y}_3)^2}{|\hat{x}-\hat{y}|^2}+1\right)
\end{split}
\end{equation*}

Hence
\begin{equation}\label{sette}
\begin{split}
&c_0\leq \frac{\Lambda}{L}|\hat{x}-\hat{y}|^{-\alpha} \left((\hat{x}_2-\hat{y}_2)^2+(\hat{x}_1-\hat{y}_1)^2\right) n_{33}\\
&+\frac{L_f}{L}|\hat{x}-\hat{y}|^{\beta-\alpha}+\frac{L_c}{L}|\hat{x}-\hat{y}|^{\beta'-\alpha}+\frac{\delta \left(2R-c_0(|\hat{x}|^2+|\hat{y}|^2)\right)}{L|\hat{x}-\hat{y}|^\alpha}\\
&= \Lambda \left((\hat{x}_2-\hat{y}_2)^2+(\hat{x}_1-\hat{y}_1)^2\right)\alpha|\hat{x}-\hat{y}|^{-2}\left((\alpha-2)\frac{(\hat{x}_3-\hat{y}_3)^2}{|\hat{x}-\hat{y}|^2}+1\right) \\
&+\Lambda L\alpha^2\frac{2}{\mu} \left((\hat{x}_2-\hat{y}_2)^2+(\hat{x}_1-\hat{y}_1)^2\right) |\hat{x}-\hat{y}|^{\alpha-4}\left(\alpha(\alpha-2)\frac{(\hat{x}_3-\hat{y}_3)^2}{|\hat{x}-\hat{y}|^2}+1\right)\\
&+\frac{L_f}{L}|\hat{x}-\hat{y}|^{\beta-\alpha}+\frac{L_c}{L}|\hat{x}-\hat{y}|^{\beta'-\alpha}+\frac{\delta}{L}|\hat{x}-\hat{y}|^{-\alpha}\left(2R-c_0|\hat{x}|^2\right).
\end{split}
\end{equation}
We remark that if the factor
$$
(\alpha-2)\frac{(\hat{x}_3-\hat{y}_3)^2}{|\hat{x}-\hat{y}|^2}+1
$$
were negative then, possibly taking $\mu$ sufficiently large, we can assume that  
\begin{equation}
\begin{split}
&0>\Lambda \left((\hat{x}_2-\hat{y}_2)^2+(\hat{x}_1-\hat{y}_1)^2\right)\alpha|\hat{x}-\hat{y}|^{-2}\left((\alpha-2)\frac{(\hat{x}_3-\hat{y}_3)^2}{|\hat{x}-\hat{y}|^2}+1\right) \\
&+\Lambda L\alpha^2\frac{2}{\mu} \left((\hat{x}_2-\hat{y}_2)^2+(\hat{x}_1-\hat{y}_1)^2\right) |\hat{x}-\hat{y}|^{\alpha-4}\left(\alpha(\alpha-2)\frac{(\hat{x}_3-\hat{y}_3)^2}{|\hat{x}-\hat{y}|^2}+1\right)
\end{split}
\end{equation}
concluding that
\begin{equation}\label{lessdifficult}
c_0\leq \frac{L_f}{L}|\hat{x}-\hat{y}|^{\beta-\alpha}+\frac{L_c}{L}|\hat{x}-\hat{y}|^{\beta'-\alpha}+\frac{\delta}{L}|\hat{x}-\hat{y}|^{-\alpha}\left(2R-c_0|\hat{x}|^2\right)
\end{equation}
and obtaining a contradiction possibly taking a larger $L$ as it explained in an analogous case later on.

Otherwise, in case
$$
(\alpha-2)\frac{(\hat{x}_3-\hat{y}_3)^2}{|\hat{x}-\hat{y}|^2}+1>0
$$
and
$$
\alpha(\alpha-2)\frac{(\hat{x}_3-\hat{y}_3)^2}{|\hat{x}-\hat{y}|^2}+1\leq 0,
$$
or, if this case occurs, 
$$
(\alpha-2)\frac{(\hat{x}_3-\hat{y}_3)^2}{|\hat{x}-\hat{y}|^2}+1>0
$$
and
$$
\alpha(\alpha-2)\frac{(\hat{x}_3-\hat{y}_3)^2}{|\hat{x}-\hat{y}|^2}+1> 0,
$$
we can obtain from (\ref{sette}), possibly fixing $\mu$ in such a way that

\begin{equation}
\begin{split}
&\Lambda \left((\hat{x}_2-\hat{y}_2)^2+(\hat{x}_1-\hat{y}_1)^2\right)\alpha|\hat{x}-\hat{y}|^{-2}\left((\alpha-2)\frac{(\hat{x}_3-\hat{y}_3)^2}{|\hat{x}-\hat{y}|^2}+1\right) \\
&=\Lambda L\alpha^2\frac{2}{\mu} \left((\hat{x}_2-\hat{y}_2)^2+(\hat{x}_1-\hat{y}_1)^2\right) |\hat{x}-\hat{y}|^{\alpha-4}\left(\alpha(\alpha-2)\frac{(\hat{x}_3-\hat{y}_3)^2}{|\hat{x}-\hat{y}|^2}+1\right),
\end{split}
\end{equation}
the following inequality 
\begin{equation}\label{estibase}
\begin{split}
&c_0\leq  2\Lambda\alpha \left((\hat{x}_2-\hat{y}_2)^2+(\hat{x}_1-\hat{y}_1)^2\right)|\hat{x}-\hat{y}|^{-2}\left((\alpha-2)\frac{(\hat{x}_3-\hat{y}_3)^2}{|\hat{x}-\hat{y}|^2}+1\right) \\
&+\frac{L_f}{L}|\hat{x}-\hat{y}|^{\beta-\alpha}+\frac{L_c}{L}|\hat{x}-\hat{y}|^{\beta'-\alpha}+\frac{\delta}{L}|\hat{x}-\hat{y}|^{-\alpha}\left(2R-c_0|\hat{x}|^2\right)
\end{split}
\end{equation}
holds.

Notice that with respect to the inequality (\ref{lessdifficult}) the worst case is given by previous inequality (\ref{estibase}). Thus we have to discuss
 the following term that appears in (\ref{estibase}):

$$
\delta\left(2R-c_0|\hat{x}|^2\right)=4\delta +(8-c_0)\delta |\hat{x}'|^2-c_0\delta \hat{x}_3^2.
$$

We also know that for every $\delta<\delta_0$
$$
\delta |\hat{x}|^2\leq 2||u||_{\infty}^2.
$$
Moreover 
$$
\delta\left(2R-c_0|\hat{x}|^2\right)\leq 4\delta +(8-c_0)\delta |\hat{x}|^2
$$
and if $c_0>8,$ we deduce that
$$
\delta\left(2R-c_0|\hat{x}|^2\right)\leq 4\delta.
$$
Thus, from (\ref{estibase}) it follows that

\begin{equation}\label{estibase2x}
\begin{split}
&c_0\leq  2\Lambda\alpha \left((\hat{x}_2-\hat{y}_2)^2+(\hat{x}_1-\hat{y}_1)^2\right)|\hat{x}-\hat{y}|^{-2}\left((\alpha-2)\frac{(\hat{x}_3-\hat{y}_3)^2}{|\hat{x}-\hat{y}|^2}+1\right) \\
&+\frac{L_f}{L}|\hat{x}-\hat{y}|^{\beta-\alpha}+\frac{L_c}{L}|\hat{x}-\hat{y}|^{\beta'-\alpha}+\frac{8\delta}{L}|\hat{x}-\hat{y}|^{-\alpha}\\
&\leq 2\Lambda\alpha+\frac{L_f}{L}|\hat{x}-\hat{y}|^{\beta-\alpha}+\frac{L_c}{L}|\hat{x}-\hat{y}|^{\beta'-\alpha}+\frac{8\delta}{L}|\hat{x}-\hat{y}|^{-\alpha}\\
&\leq 2\Lambda\alpha+2^{\frac{\beta}{\alpha}-1}\left(\frac{L_f}{L^{\frac{\beta}{\alpha}}}||u||_\infty^{\frac{\beta}{\alpha}-1}+\frac{L_c}{L^{\frac{\beta'}{\alpha}}}||u||_\infty^{\frac{\beta'}{\alpha}-1}\right)+\frac{8\delta}{L}|\hat{x}-\hat{y}|^{-\alpha}
\end{split}
\end{equation}

and for $\delta\to 0$ we conclude also

\begin{equation}\label{estibase2}
\begin{split}
&c_0\leq  2\Lambda\alpha+2^{\frac{\beta}{\alpha}-1}\left(\frac{L_f}{L^{\frac{\beta}{\alpha}}}||u||_\infty^{\frac{\beta}{\alpha}-1}+\frac{L_c}{L^{\frac{\beta'}{\alpha}}}||u||_\infty^{\frac{\beta'}{\alpha}-1}\right).
\end{split}
\end{equation}

Hence by taking $L$ sufficiently large and $\alpha$ sufficiently small, we get a contradiction with the positivity of $c_0.$ In particular we need that
$$
\alpha<\frac{c_0}{2\Lambda}.
$$

In case $\tilde{A}+2K\tilde{I}<B,$ then 
$$
\lambda\mbox{Tr}(\tilde{B}-\tilde{A}-2K\tilde{I})\leq F(\tilde{B})-F(\tilde{A}+2\delta\tilde{I})\leq \Lambda \mbox{Tr}(\tilde{B}-\tilde{A}-2\delta\tilde{I}),
$$
so that
$$
-\Lambda\mbox{Tr}(\tilde{B}-\tilde{A}-2\delta\tilde{I})\leq F(\tilde{A}+2\delta\tilde{I})-F(\tilde{B})\leq -\lambda \mbox{Tr}(\tilde{B}-\tilde{A}-2\delta\tilde{I})
$$
and, as a consequence, we get that $$F(\tilde{A}+2\delta\tilde{I})-F(\tilde{B})\leq -\lambda \mbox{Tr}(\tilde{B}-\tilde{A}-2\delta\tilde{I})<0.$$ The contradiction follows
from (\ref{mainine}) since we immediately obtain
\begin{equation}
\begin{split}
&c_0L|\hat{x}-\hat{y}|^\alpha \leq F(\tilde{A}+2\delta\tilde{I})-F(B)+f(\hat{y})-f(\hat{x})+u(\hat{y})(c(\hat{y})-c(\hat{x}))\\
&\leq -\lambda \mbox{Tr}(\tilde{B}-\tilde{A}-2\delta\tilde{I})+ L_f|\hat{x}-\hat{y}|^\beta+L_c|\hat{x}-\hat{y}|^{\beta'}\\
&\leq L_f|\hat{x}-\hat{y}|^\beta+L_c|\hat{x}-\hat{y}|^{\beta'}.
\end{split}
\end{equation}
Moreover we  obtain
\begin{equation}
\begin{split}
&c_0\leq \frac{L_f}{L}|\hat{x}-\hat{y}|^{\beta-\alpha}+\frac{L_c}{L}|\hat{x}-\hat{y}|^{\beta'-\alpha}
\end{split}
\end{equation}
that implies a contradiction whenever $\beta,\beta'\geq \alpha$ and $c_0> \frac{L_f}{L}+\frac{L_c}{L}.$

\end{proof}

\section{A non-intrinsic approach}\label{nonintrinsica}
In this section we discuss the operator defined as in Definition \ref{definition1H}
\begin{lemma}
Let us be given three symmetric matrices $P,$ $S_1$ $S_2$ such that $S_1\leq S_2$ and $P\geq 0.$ Then
$$
PS_1P\leq PS_2P.
$$  
\end{lemma}
 \begin{corollary}
 Let us suppose that $A$ and $B$ are  symmetric matrices, $P\geq 0$ for every $x\in \mathbb{G}$ and
 \begin{equation}
 \begin{split}
 \left[
 \begin{array}{rr}
 A,&0\\
0,&-B
 \end{array}
 \right]\leq D^2\phi+\frac{(D^2\phi)^2}{\mu}
 \end{split}
 \end{equation}
 Then
  \begin{equation}
 \begin{split}
& \left[
 \begin{array}{rr}
\sqrt{P}(x) A\sqrt{P}(x),&0\\
0,&-\sqrt{P}(y)B\sqrt{P}(y)
 \end{array}
 \right]\\
 &\leq  \left[
 \begin{array}{rr}
\sqrt{P}(x),&0\\
0,&-\sqrt{P}(y)
 \end{array}
 \right]\left(D^2\phi+\frac{(D^2\phi)^2}{\mu}\right)\left[
 \begin{array}{rr}
\sqrt{P}(x),&0\\
0,&-\sqrt{P}(y)
 \end{array}
 \right]
 \end{split}
 \end{equation}
 \end{corollary}
 At the same time it is true the following result.
  \begin{corollary}\label{corol9}
 Let us suppose that $A,$  $B, $ $N$ and 
 $P$ are  symmetric matrices, $P\geq 0$ for every $x\in \mathbb{G}$ and
 \begin{equation}
 \begin{split}
 \left[
 \begin{array}{rr}
 A,&0\\
0,&-B
 \end{array}
 \right]\leq \left[
 \begin{array}{rr}
N,&-N\\
-N,&N
 \end{array}
 \right]
 \end{split}
 \end{equation}
 Then
  \begin{equation}
 \begin{split}
& \left[
 \begin{array}{rr}
\sqrt{P}(x) A\sqrt{P}(x),&0\\
0,&-\sqrt{P}(y)B\sqrt{P}(y)
 \end{array}
 \right]\\
 &\leq  \left[
 \begin{array}{rr}
\sqrt{P}(x),&0\\
0,&\sqrt{P}(y)
 \end{array}
 \right]\left[
 \begin{array}{rr}
N,&-N\\
-N,&N
 \end{array}
 \right]\left[
 \begin{array}{rr}
\sqrt{P}(x),&0\\
0,&\sqrt{P}(y)
 \end{array}
 \right]
 \end{split}
 \end{equation}
 \end{corollary}
 \begin{lemma}
Let be $N$ be a given symmetric matrix and $P\geq 0.$ Let us denote
  \begin{equation}
 \begin{split}
 K=\left[
 \begin{array}{rr}
\sqrt{P}(x),&0\\
0,&\sqrt{P}(y)
 \end{array}
 \right]\left[
 \begin{array}{rr}
N,&-N\\
-N,&N
 \end{array}
 \right]\left[
 \begin{array}{rr}
\sqrt{P}(x),&0\\
0,&\sqrt{P}(y)
 \end{array}
 \right]
 \end{split}
 \end{equation}
 Then for every $\xi_1,\xi_2\in \mathbb{R}^n$
 $$
 \langle K[\xi_1,\xi_2]^T,[\xi_1,\xi_2]\rangle=\langle N(\sqrt{P(x)}\xi_1-\sqrt{P(y)}\xi_2),(\sqrt{P(x)}\xi_1-\sqrt{P(y)}\xi_2)\rangle
 $$
 and
   \begin{equation}
 \begin{split}
& \langle \sqrt{P}(x) A\sqrt{P}(x)\xi_1,\xi_1\rangle-\langle  \sqrt{P}(y) B\sqrt{P}(y)\xi_2,\xi_2\rangle\\
 &\leq  \langle N(\sqrt{P(x)}\xi_1-\sqrt{P(y)}\xi_2),(\sqrt{P(x)}\xi_1-\sqrt{P(y)}\xi_2)\rangle.
  \end{split}
 \end{equation}
 \end{lemma}
 \begin{proof}
 Let $[\xi_1,\xi_2]\in \mathbb{R}^{2m},$ then, by the symmetric properties,  applying the definition and denoting $$
v=\left[
 \begin{array}{rr}
\sqrt{P}(x),&0\\
0,&\sqrt{P}(y)
 \end{array}
 \right] [\xi_1,\xi_2]^T=\left[
 \begin{array}{r}
\sqrt{P}(x)\xi_1\\
\sqrt{P}(y)\xi_2
 \end{array}
 \right]
$$ 
we get that
 
   \begin{equation*}
 \begin{split}
&\langle\left[
 \begin{array}{rr}
N,&-N\\
-N,&N
 \end{array}
 \right]v,v\rangle =\langle N\sqrt{P(x)}\xi_1,\sqrt{P(x)}\xi_1\rangle-\langle N\sqrt{P(y)}\xi_2,\sqrt{P(x)}\xi_1\rangle\\
 &-\langle N\sqrt{P(x)}\xi_1,\sqrt{P(y)}\xi_2\rangle+\langle N\sqrt{P(y)}\xi_2,\sqrt{P(y)}\xi_2\rangle\\
 &=\langle N(\sqrt{P(x)}\xi_1-\sqrt{P(y)}\xi_2),\sqrt{P(x)}\xi_1-\sqrt{P(y)}\xi_2\rangle\\
 &\leq |N| |\sqrt{P(x)}\xi_1-\sqrt{P(y)}\xi_2|^2. \end{split}
 \end{equation*}

 \end{proof}

 \begin{corollary}\label{matrix1}
 Let us suppose that $A,$  $B,$  $M$ and $P$ are  symmetric matrices $m\times m$, $P\geq 0$ for every $x\in \mathbb{G}\equiv\mathbb{R}^m$ and
 \begin{equation}
 \begin{split}
 \left[
 \begin{array}{rr}
 A,&0\\
0,&-B
 \end{array}
 \right]\leq \left[
 \begin{array}{rr}
N,&-N\\
-N,&N
 \end{array}
 \right].
 \end{split}
 \end{equation}
 Then
   \begin{equation}
 \begin{split}
& \mbox{Tr}(P(x)A-P(y) B)\leq  m|N||\sqrt{P(x)}-\sqrt{P(y)}|^2.
  \end{split}
 \end{equation}
 \end{corollary}
 \begin{proof}
 Let $\xi_1=\xi_2,$  and $|\xi_1|=1$ then
   \begin{equation}
 \begin{split}
&\langle (\sqrt{P(x)} A\sqrt{P(x)}-\sqrt{P(y)} B\sqrt{P(y)})\xi_1,\xi_1\rangle\\
&=\langle \sqrt{P(x)} A\sqrt{P(x)}\xi_1,\xi_1\rangle-\langle  \sqrt{P(y)} B\sqrt{P(y)}\xi_1,\xi_1\rangle\\
 &\leq  \langle N(\sqrt{P(x)}-\sqrt{P(y)})\xi_1,(\sqrt{P(x)}-\sqrt{P(y)})\xi_1\rangle\leq |N||\sqrt{P(x)}-\sqrt{P(y)}|^2
  \end{split}
 \end{equation}
 Thus, by taking an orthonormal basis given by the eigenvectors of $$\sqrt{P(x)} A\sqrt{P(x)}-\sqrt{P(y)} B\sqrt{P(y)},$$ we obtain our thesis. Indeed
  \begin{equation}
 \begin{split}
&\mbox{Tr}(P(x)A-P(y) B)=\mbox{Tr}(P(x)A)-\mbox{Tr}(P(y) B)\\
&=\mbox{Tr}(\sqrt{P(x)} A\sqrt{P(x)})-\mbox{Tr}(\sqrt{P(y)} B\sqrt{P(y)})\\
&=\mbox{Tr}(\sqrt{P(x)} A\sqrt{P(x)}-\sqrt{P(y)} B\sqrt{P(y)})\leq m|N||\sqrt{P(x)}-\sqrt{P(y)}|^2.
  \end{split}
 \end{equation}
 \end{proof}

\begin{remark}
In the Heisenberg group $\mathbb{H}^1,$ $m=3.$
\end{remark}
\begin{remark}
It is now interesting to remark that we can not expect that there exists, in general, an eigenvector $w$ such that 
  \begin{equation}
 \begin{split}
 \left[
 \begin{array}{rr}
\sqrt{P(x)},&0\\
0,&\sqrt{P(y)}
 \end{array}
 \right]\left[
 \begin{array}{rr}
N,&-N\\
-N,&N
 \end{array}
 \right]\left[
 \begin{array}{rr}
\sqrt{P(x)},&0\\
0,&\sqrt{P(y)}
 \end{array}
 \right]w=\lambda w
 \end{split}
 \end{equation}
 and such that its eigenvalue is strictly negative for every couple of $x$ and $y.$ Indeed, if it exists, then it should be also true that
 $$
 \sqrt{P(x)}\xi_1-\sqrt{P(y)}\xi_2\in\mathbb{R}\{\frac{x-y}{|x-y|}\},
 $$
 possibly  taking some appropriate vectors $\xi_1\in \mathbb{R}^m,$ $\xi_2\in \mathbb{R}^m.$ Nevertheless, considering the degenerateness of $P,$ for every $x\in \mathbb{G},$ it is possible to prove that  given $\frac{x-y}{|x-y|},$  for some particular $x$ and $y,$ it is not possible to find $\xi_1$ and $\xi_2$ such that 
  $$
 \sqrt{P(x)}\xi_1-\sqrt{P(y)}\xi_2\in\mathbb{R}\{\frac{x-y}{|x-y|}\}.
 $$
 For example in $\mathbb{H}^1$ in the Heisenberg group, where $m=3,$ if $x=(0,x_3)$ and $y=(0,y_3)$  and $x_3\not=y_3,$ then for every $\xi_1,\xi_2\in\mathbb{R}^3\setminus\{(0,0,0)\}$
 $$
 \sqrt{P(x)}\xi_1-\sqrt{P(y)}\xi_2\not =(0,0,\pm 1).
 $$
 \end{remark}
 \subsection{The square root matrix in the Heisenberg group}
 The square root of $P_{\mathbb{H}^1}(x)$ is
  \begin{equation}
\sqrt{P_{\mathbb{H}^1}(x)}
 =\left[
 \begin{array}{lll}
\frac{ x_2^2+\frac{x_1^2}{\sqrt{1+4(x_1^2+x_2^2)}}}{x_1^2+x_2^2},&\frac{x_1x_2(1-\frac{1}{\sqrt{1+4(x_1^2+x_2^2)}})}{x_1^2+x_2^2},&\frac{2x_2}{\sqrt{1+4(x_1^2+x_2^2)}}\\
\frac{x_1x_2(1-\frac{1}{\sqrt{1+4(x_1^2+x_2^2)}})}{x_1^2+x_2^2},&\frac{ x_1^2+\frac{x_2^2}{\sqrt{1+4(x_1^2+x_2^2)}}}{x_1^2+x_2^2},&-\frac{2x_1}{\sqrt{1+4(x_1^2+x_2^2)}}\\
\frac{2x_2}{\sqrt{1+4(x_1^2+x_2^2)}},&-\frac{2x_1}{\sqrt{1+4(x_1^2+x_2^2)}},&\frac{4(x_1^2+x_2^2)}{\sqrt{1+4(x_1^2+x_2^2)}}
 \end{array}
 \right]
 \end{equation}

On the other hand  by an elementary algebraic manipulation of $\sqrt{P_{\mathbb{H}^1}(x)}$ we get 
  \begin{equation}
  \begin{split}
&\sqrt{P_{\mathbb{H}^1}(x)}\\
 &=\left[
 \begin{array}{lll}
1-\frac{4x_1^2}{(1+\sqrt{1+4(x_1^2+x_2^2)})\sqrt{1+4(x_1^2+x_2^2)}},&\frac{4x_1x_2}{(1+\sqrt{1+4(x_1^2+x_2^2)})\sqrt{1+4(x_1^2+x_2^2)}},&\frac{2x_2}{\sqrt{1+4(x_1^2+x_2^2)}}\\
\frac{4x_1x_2}{(1+\sqrt{1+4(x_1^2+x_2^2)})\sqrt{1+4(x_1^2+x_2^2)}},&1-\frac{4x_2^2}{(1+\sqrt{1+4(x_1^2+x_2^2)})\sqrt{1+4(x_1^2+x_2^2)}},&-\frac{2x_1}{\sqrt{1+4(x_1^2+x_2^2)}}\\
\frac{2x_2}{\sqrt{1+4(x_1^2+x_2^2)}},&-\frac{2x_1}{\sqrt{1+4(x_1^2+x_2^2)}},&\frac{4(x_1^2+x_2^2)}{\sqrt{1+4(x_1^2+x_2^2)}}
 \end{array}
 \right]
 \end{split}
 \end{equation}
 that is, in particular, we can conclude that $\sqrt{P_{\mathbb{H}^1}(x)}$ is still a $C^{\infty}$ matrix and moreover the following result holds.

 \begin{lemma}\label{estimateonsquareroot}
There exists $C_2>0$ such that for every $\gamma>0$  and for every  $x,y\in \mathbb{H}^1,$ such that $\gamma<|x-y|<\gamma^{-1}$  and $|x|\to\infty$ and $|y|\to\infty,$ then 
 $$
 |\sqrt{P_{\mathbb{H}^1}(x)}-\sqrt{P_{\mathbb{H}^1}(y)}|= |\sqrt{P_{\mathbb{H}^1}(x')}-\sqrt{P_{\mathbb{H}^1}(y')}|\leq C_2|x'-y'|.
 $$
 \end{lemma}
 \begin{proof}
 Let us consider for simplicity only the unbounded coefficient  of $\sqrt{P_{\mathbb{H}^1}(x)},$ as $|x|\to\infty,$ given by $\frac{4(x_1^2+x_2^2)}{\sqrt{1+4(x_1^2+x_2^2)}}.$ Let us denote $|x'|=\rho$ and $|y'|=r.$ Then
 \begin{equation*}
 \begin{split}
 &\frac{4(x_1^2+x_2^2)}{\sqrt{1+4(x_1^2+x_2^2)}}-\frac{4(y_1^2+y_2^2)}{\sqrt{1+4(y_1^2+y_2^2)}}=\frac{4\rho^2}{\sqrt{1+4\rho^2}}-\frac{4r^2}{\sqrt{1+4r^2}}\\
 &=\frac{4(\rho^2-r^2)}{\sqrt{1+4\rho^2}}+4r^2\left(\frac{1}{\sqrt{1+4\rho^2}}-\frac{1}{\sqrt{1+4r^2}}\right)\\
 &=4(\rho-r)\frac{(\rho+r)}{\sqrt{1+4\rho^2}}+4r^2
 \left(\frac{\sqrt{1+4r^2}-\sqrt{1+4\rho^2}}{\sqrt{1+4\rho^2}\sqrt{1+4r^2}}\right)\\
 &=4(\rho-r)\frac{(\rho+r)}{\sqrt{1+4\rho^2}}+\frac{4r^2}{\sqrt{1+4\rho^2}\sqrt{1+4r^2}}\frac{4(r^2-\rho^2}{\sqrt{1+4\rho^2}+\sqrt{1+4r^2}}\\
 &=4(\rho-r)(R_1+R_2)
 \end{split}
 \end{equation*}
 where
 $$
 R_1=\frac{(\rho+r)}{\sqrt{1+4\rho^2}}
 $$
 and
 $$
 R_2=\frac{4r^2(\rho+r)}{\sqrt{1+4\rho^2}\sqrt{1+4r^2}(\sqrt{1+4\rho^2}+\sqrt{1+4r^2})}.
 $$
 Hence $R_1+R_2$ are uniformly bounded whenever $\gamma<|x-y|<\gamma^{-1}$  and $|x|\to\infty$ and $|y|\to\infty,$ let say by $C_2$
 
 Thus
 \begin{equation*}
 \begin{split}
 &|\frac{4(x_1^2+x_2^2)}{\sqrt{1+4(x_1^2+x_2^2)}}-\frac{4(y_1^2+y_2^2)}{\sqrt{1+4(y_1^2+y_2^2)}}|\leq 4|\rho-r|(R_1+R_2)\leq C_2 |x'-y'|.
 \end{split}
 \end{equation*}
 \end{proof}
 Moreover the following global result is still true

 As a consequence we get the following Corollary.
  \begin{corollary}\label{osservazione2}
  Let us suppose that $A,$  $B, $ $N$ and 
  are  symmetric matrices $3\times 3$ such that,  for every $x,y\in \mathbb{H}^1,$  
 \begin{equation}
 \begin{split}
 \left[
 \begin{array}{rr}
 A,&0\\
0,&-B
 \end{array}
 \right]\leq \left[
 \begin{array}{rr}
N,&-N\\
-N,&N
 \end{array}
 \right].
 \end{split}
 \end{equation}
 There exists $C>0$ such that for every  $x,y\in \mathbb{H}^1,$  $\gamma<|x-y|<\gamma^{-1}$  and $|x|\to\infty$ and $|y|\to\infty,$  
  \begin{equation}
 \begin{split}
& \mbox{Tr}(P_{\mathbb{H}^1}(x)A-P_{\mathbb{H}^1}(y) B)\leq C|N||x'-y'|^2.
  \end{split}
 \end{equation}
 \end{corollary}
 \begin{proof}
 Recalling Corollary \ref{matrix1} we get 
  \begin{equation}
 \begin{split}
& \mbox{Tr}(P_{\mathbb{H}^1}(x)A-P_{\mathbb{H}^1}(y) B)\leq  3|N||\sqrt{P_{\mathbb{H}^1}(x)}-\sqrt{P_{\mathbb{H}^1}(y)}|^2.
  \end{split}
 \end{equation}
 Thus invoking the previous estimate of $\sqrt{P_{\mathbb{H}^1}}$ contained in Lemma \ref{estimateonsquareroot} we conclude that there exists a positive constant $C=3C_2$ such that for every $x,y\in B_1(\bar{x}),$
  \begin{equation}
 \begin{split}
& \mbox{Tr}(P_{\mathbb{H}^1}(x)A-P_{\mathbb{H}^1}(y) B)\leq C|N| |x'-y'|^2.
  \end{split}
 \end{equation}
 \end{proof}
 \begin{corollary}
 There exists a positive constant $C$ such that  if $A$ and $B$ are the matrices determined by the Theorem of the sums applied to the function $w,$  then
 $$
  \mbox{Tr}(P_{\mathbb{H}^1}(\hat{x})A-P_{\mathbb{H}^1}(\hat{y}) B)\leq C\left( L\alpha |\hat{x}-\hat{y}|^{\alpha}+\frac{2}{\mu}L^2\alpha^2 |\hat{x}-\hat{y}|^{2\alpha-2}\right).
 $$
 \end{corollary}
 \begin{proof}
 Recalling the Remark \ref{osservazione1} and Corollary \ref{osservazione2} we conclude that 
  \begin{equation}
 \begin{split}
& \mbox{Tr}(P_{\mathbb{H}^1}(\hat{x})A-P_{\mathbb{H}^1}(\hat{y}) B)\\
&\leq C\left( L\alpha |\hat{x}-\hat{y}|^{\alpha-2}+\frac{2}{\mu}L^2\alpha^2 |\hat{x}-\hat{y}|^{2(\alpha-2)}\right) |\hat{x}'-\hat{y}'|^2\\
&\leq C\left( L\alpha |\hat{x}-\hat{y}|^{\alpha}+\frac{2}{\mu}L^2\alpha^2 |\hat{x}-\hat{y}|^{2\alpha-2}\right) 
  \end{split}
 \end{equation}
 \end{proof}
 With these estimates we can replicate the proof of the main theorem already given for operators satisfying Definition \ref{definition2H} covering also the case of operators defined like in Definition \ref{definition1H}.

\end{document}